\documentclass[11pt,reqno]{amsart}

\usepackage[T1]{fontenc}
\usepackage{lmodern}
\usepackage{amsmath,amssymb}
\usepackage{array,booktabs,cite}
\usepackage[hidelinks]{hyperref}

\hypersetup{
  pdftitle={Linear Corrector Estimates and Corrected Pressure Convergence
    for a Two-Parameter Hyperbolic Relaxation of the Incompressible
    Navier--Stokes Equations},
  pdfauthor={Leyang Wang and Wenlong Lin},
  pdfsubject={Analysis of partial differential equations},
  pdfkeywords={incompressible Navier--Stokes equations, hyperbolic relaxation,
    artificial compressibility, linear corrector, corrected pressure convergence,
    singular perturbation}
}

\allowdisplaybreaks
\newtheorem{theorem}{Theorem}[section]
\newtheorem{proposition}[theorem]{Proposition}
\newtheorem{lemma}[theorem]{Lemma}
\newtheorem{corollary}[theorem]{Corollary}
\theoremstyle{definition}

\theoremstyle{remark}
\newtheorem{remark}[theorem]{Remark}
\numberwithin{equation}{section}

\newcommand{\T}{\mathbb T}
\newcommand{\R}{\mathbb R}
\newcommand{\eps}{\epsilon}
\newcommand{\dd}{\,\mathrm d}
\newcommand{\Div}{\operatorname{div}}
\newcommand{\Curl}{\operatorname{curl}}

\newcommand{\ip}[2]{\left\langle #1,#2\right\rangle}
\newcommand{\norm}[2]{\left\|#1\right\|_{#2}}

\begin{document}

\title[Linear correctors and corrected pressure convergence]
{Linear Corrector Estimates and\\
Corrected Pressure Convergence\\
for a Two-Parameter Hyperbolic Relaxation\\
of the Incompressible Navier--Stokes Equations}

\author{Leyang Wang\textsuperscript{*}}

\thanks{\textsuperscript{*}Corresponding author: Leyang Wang.}

\author{Wenlong Lin}
\address[Both authors]{Mathematics and Science College, Shanghai Normal University,
Shanghai 200234, China}
\email[Leyang Wang]{1000446516@smail.shnu.edu.cn}
\email[Wenlong Lin]{1000486631@smail.shnu.edu.cn}
\renewcommand{\shortauthors}{Leyang Wang and Wenlong Lin}

\date{October 5, 2026}

\subjclass[2020]{Primary 35Q30; Secondary 35L45, 35B40, 76D05}

\keywords{Incompressible Navier--Stokes equations, hyperbolic relaxation,
artificial compressibility, linear corrector, corrected pressure convergence,
singular perturbation}

\begin{abstract}
We study a two-parameter first-order hyperbolic relaxation approximation
of the incompressible Navier--Stokes equations on the two-dimensional
torus, under the derivative-level large-perturbation assumptions of
Huang, Rohde and Zhang.  Building on the auxiliary-system approach used
in earlier pressure-convergence results, we introduce a non-autonomous
linear corrector that carries the complete initial discrepancy, the
linearized convection and the smooth consistency forcing.  This
corrector describes the full linear response; it is not asserted to be a
purely acoustic component.  An exact compensated energy identity,
including a time integration by parts in the linearized convection
term, yields an integrated gradient estimate for the velocity corrector.
The two-dimensional Ladyzhenskaya inequality then bounds its quadratic
self-interaction by $O(\delta^2)$ in squared space-time $L^2$ norm.
Consequently, the nonlinear remainder has weighted energy
$O(\delta^2)$, and its pressure and velocity components satisfy the
$L^\infty(0,T;L^2)$ bounds $C_T\delta/\sqrt{\eps}$ and $C_T\delta$,
respectively.  In the sufficient parameter window
$\delta^2\ll\eps\leq\mu_*\delta$, the pressure after subtraction of
the linear corrector converges strongly to the incompressible pressure.
A mean-zero Fourier-mode construction, transferred to the nonlinear
system by the remainder estimate, shows that the uncorrected pressure
need not converge even on time intervals bounded away from zero.
\end{abstract}

\maketitle

\section{Introduction and main result}
\label{sec:introduction}

\subsection{The relaxation model}

Let $\T^2=\R^2/\mathbb Z^2$ and fix $T>0$.  We consider a smooth
solution $(p,u)$ of the incompressible Navier--Stokes equations
\begin{equation}
\label{eq:NS}
\left\{
\begin{aligned}
    \Div u&=0,\\
    \partial_tu+\Div(u\otimes u)+\nabla p&=\Delta u,\\
    u|_{t=0}&=u_0
\end{aligned}
\right.
\qquad\text{on }[0,T]\times\T^2.
\end{equation}
The viscosity has been normalized to one, and the pressure is normalized
by $\int_{\T^2}p(t,x)\,\dd x=0$.  Define the stress-like tensor
\begin{equation}
\label{eq:limit-stress}
    U:=u\otimes u-\nabla u.
\end{equation}
The momentum equation in \eqref{eq:NS} is then equivalently
\begin{equation}
\label{eq:limit-momentum-stress}
    \partial_tu+\Div U+\nabla p=0.
\end{equation}

For two small parameters $\eps,\delta>0$, the approximation under
consideration is
\begin{equation}
\label{eq:relaxation-system}
\left\{
\begin{aligned}
    \eps\partial_t p^{\eps,\delta}
        +\Div u^{\eps,\delta}&=0,\\
    \partial_tu^{\eps,\delta}
        +\Div U^{\eps,\delta}
        +\nabla p^{\eps,\delta}&=0,\\
    \delta\partial_tU^{\eps,\delta}
        +\nabla u^{\eps,\delta}
        &=u^{\eps,\delta}\otimes u^{\eps,\delta}
          -U^{\eps,\delta}.
\end{aligned}
\right.
\end{equation}
Here $p^{\eps,\delta}$ is scalar-valued,
$u^{\eps,\delta}$ is $\R^2$-valued, and
$U^{\eps,\delta}$ is $\R^{2\times2}$-valued.  The first equation is an
artificial-compressibility equation, whereas the last equation relaxes
the tensor $U^{\eps,\delta}$ toward
$u^{\eps,\delta}\otimes u^{\eps,\delta}-\nabla u^{\eps,\delta}$.
Thus \eqref{eq:relaxation-system} formally reduces to
\eqref{eq:NS} as $(\eps,\delta)\to(0,0)$.

Artificial-compressibility approximations go back to Chorin and Temam
\cite{Chorin1967,Chorin1968,Temam1969a,Temam1969b}.  Hyperbolic
relaxations of viscous fluid equations have subsequently been studied
from several viewpoints; see, for example,
\cite{BrenierNataliniPuel2004,PaicuRaugel2007,RackeSaal2012,
Hachicha2014,IlyinRykovZelik2018}.  General singular-limit and
relaxation frameworks can be found in
\cite{JinXin1995,KlainermanMajda1981,Yong1999,Yong2001}, while
kinetic relaxation limits are represented, for instance, by
\cite{Bianchini2019}.  Analytical results for artificial
compressibility include
\cite{DonatelliMarcati2010,Kagei2021,Temam2001}.  The combined
two-parameter system \eqref{eq:relaxation-system} was introduced and
analyzed in \cite{HuangRohdeYongZhang2025}.  The more recent work
\cite{HuangRohdeZhang2026} proves pressure convergence under stronger
initial preparation, using an intermediate affine system, and robust
velocity estimates in a derivative-level large-perturbation regime.  In two dimensions, those
estimates allow the derivatives of the initial velocity discrepancy to
remain of order one, or even to grow with the singular parameters.

The pressure behaves differently.  The standard symmetric-hyperbolic
energy controls $\sqrt{\eps}(p^{\eps,\delta}-p)$, which does not by
itself imply convergence of the unweighted pressure.  Our purpose is to
estimate the nonlinear error remaining after the full linear response
has been subtracted.  The resulting pressure correction depends on the
exact limiting solution and is an analytical asymptotic decomposition.
It is not a pressure postprocessing procedure based solely on the
relaxation solution.

\subsection{Relation to earlier results and contribution}
\label{subsec:comparison}

The use of a linear auxiliary system with the full relaxation initial
data originates in \cite{HuangRohdeYongZhang2025} and is extended to
three dimensions in \cite{HuangRohdeZhang2026}.  Subtracting the limit
solution from that auxiliary system already produces a
corrector--remainder splitting with zero initial data for the nonlinear
remainder.  We use this architecture here.  The present corrector also
retains the linearized convection $B_u(a)$, rather than freezing the
quadratic tensor at $u\otimes u$.

For clarity, the relevant results have the following different scopes.
The preparation conditions of the earlier pressure theorems are not
replaced by the assumptions of the present paper for the same
uncorrected convergence conclusion.
\begin{center}
\small
\begin{tabular}{@{}p{0.29\textwidth}p{0.65\textwidth}@{}}
\toprule
Result & Preparation and convergence statement \\
\midrule
\cite{HuangRohdeYongZhang2025}, pressure results on $\T^2$
& Prepared data; an affine auxiliary system is used to prove
convergence of the uncorrected pressure. \\[3pt]
\cite[Theorem~2.5]{HuangRohdeZhang2026}, on $\T^3$
& Weighted $H^1$ initial discrepancies of order $\eps+\delta$,
with additional derivative preparation; uncorrected pressure
convergence when $\delta/\sqrt{\eps}\to0$. \\[3pt]
\cite[Theorem~2.8]{HuangRohdeZhang2026}, on $\T^2$
& The derivative-level assumptions used below; velocity convergence
and uniform curl--divergence estimates. \\[3pt]
Present result, on $\T^2$
& The same derivative-level assumptions as that baseline theorem;
corrected pressure error $O(\delta/\sqrt{\eps})$ and corrected
velocity error $O(\delta)$, with pressure convergence when
$\delta^2/\eps\to0$ and $\eps\leq\mu_*\delta$. \\
\bottomrule
\end{tabular}
\end{center}

The compensated combinations used below are also related to existing
ones.  The damped-wave combination $a+\delta\partial_ta$ is familiar
from modulated relaxation energies
\cite{BrenierNataliniPuel2004,HuangRohdeZhang2026}.  Moreover,
$Y/\sqrt{\eps}=\pi+\delta\partial_t\pi+\delta\partial_tp$ has the
same algebraic form as the pressure combination used in
\cite[Section~3.3]{HuangRohdeZhang2026}.  We do not claim novelty for
these combinations alone, or for the restriction
$\delta/\sqrt{\eps}\to0$.

The technical contribution is the exact compensated identity for the
non-autonomous corrector with $B_u(a)$, together with the resulting
integrated estimate $\int_0^T\|\nabla a\|_{L^2}^2\,\dd t\leq C_T\delta$.
The convection term is handled without estimating $\nabla X$.
This estimate makes the corrector self-interaction a source of size
$O(\delta^2)$ in squared space-time $L^2$ norm, and yields the
next-order nonlinear remainder under the baseline assumptions.

In transport modeling, the two parameters represent artificial
compressibility and stress relaxation.  The result quantifies how
accurately the nonlinear relaxation error is described by its linear
response when derivatives of the initial discrepancy are not small.
This is an error-analysis statement for the relaxation model; no
numerical implementation of the correction is proposed here.

\subsection{Regularity and constants}
\label{subsec:regularity}

We work with smooth limiting solutions and $C^\infty$ periodic
relaxation initial data for each fixed pair $(\eps,\delta)$.
No uniform bound on the higher smooth norms of those initial data is
assumed beyond the parameter-weighted bounds stated below.
For the limit solution, the following finite collection of norms is a
sufficient, nonoptimized regularity requirement for all the estimates:
\begin{equation}
\label{eq:limit-regularity}
\begin{aligned}
 u&\in C([0,T];H^6)\cap C^1([0,T];H^4)
                       \cap C^2([0,T];H^2),\\
 p&\in C([0,T];H^5)\cap C^1([0,T];H^3),
 \qquad U=u\otimes u-\nabla u.
\end{aligned}
\end{equation}
Write $\|f\|_{CH^m}:=\sup_{0\leq t\leq T}\|f(t)\|_{H^m}$ and set
\begin{equation}
\label{eq:limit-norm-constant}
\begin{aligned}
 M_T:=1+{}&\|u\|_{CH^6}+\|\partial_tu\|_{CH^4}
 +\|\partial_{tt}u\|_{CH^2}\\
 &+\|p\|_{CH^5}+\|\partial_tp\|_{CH^3}.
\end{aligned}
\end{equation}
These assumptions give $\partial_tU\in C([0,T];H^3)$ and bound the
following quantities in terms of $M_T$:
\[
 \|u\|_{L^\infty W^{1,\infty}}+
 \|\partial_tu\|_{L^\infty W^{1,\infty}},\qquad
 \|\partial_tp\|_{L^\infty L^2}+
 \|\partial_tU\|_{L^\infty L^2}+
 \|\partial_{tt}u\|_{L^\infty L^2}.
\]
Here the time interval is $[0,T]$.  Constants denoted by $C_T$, as well as the smallness
thresholds in the main theorem, may depend on $C_0$, $T$, $M_T$ and the
fixed torus, but not on $\eps$, $\delta$ or unprescribed higher norms
of the approximate initial data.  Smoothness justifies the identities
below; the estimates themselves use only this finite set of limit norms.

\subsection{Initial assumptions}
\label{subsec:initial-assumptions}

Write
\[
    p_0^{\eps,\delta}=p^{\eps,\delta}(0),\qquad
    u_0^{\eps,\delta}=u^{\eps,\delta}(0),\qquad
    U_0^{\eps,\delta}=U^{\eps,\delta}(0),
\]
and let $p_0=p(0)$ and $U_0=U(0)$.  We impose the same
two-dimensional derivative-level assumptions as in the baseline
large-perturbation theorem of \cite{HuangRohdeZhang2026}.  Namely, for
some $C_0\geq1$ independent of $\eps$ and $\delta$,
\begin{equation}
\label{eq:main-initial-velocity-L2}
    \norm{u_0^{\eps,\delta}-u_0}{L^2}^2\leq C_0\delta,
\end{equation}
\begin{equation}
\label{eq:main-initial-curl-div}
\begin{aligned}
    \norm{\Curl u_0^{\eps,\delta}}{L^2}^2
      +\delta\norm{\Curl u_0^{\eps,\delta}}{H^1}^2
        &\leq C_0,\\
    \norm{\Div u_0^{\eps,\delta}}{L^2}^2
      +\delta\norm{\Div u_0^{\eps,\delta}}{H^1}^2
        &\leq C_0\frac{\eps}{\delta}.
\end{aligned}
\end{equation}
The initial pressure and stress satisfy
\begin{equation}
\label{eq:main-initial-pressure}
    \eps\norm{p_0^{\eps,\delta}}{L^2}^2
    +\delta^2\norm{p_0^{\eps,\delta}}{H^1}^2
    +\delta^3\norm{p_0^{\eps,\delta}}{H^2}^2
    \leq C_0\delta
\end{equation}
and
\begin{equation}
\label{eq:main-initial-stress}
    \norm{U_0^{\eps,\delta}}{L^2}^2
    +\delta\norm{U_0^{\eps,\delta}}{H^1}^2
    +\delta^2\norm{\Div U_0^{\eps,\delta}}{H^1}^2
    \leq C_0.
\end{equation}
The terms $\delta^3\|p_0^{\eps,\delta}\|_{H^2}^2$ and
$\delta^2\|\Div U_0^{\eps,\delta}\|_{H^1}^2$ are required by
the baseline existence and curl--divergence theorem.  The subsequent
corrector and remainder estimates use its conclusions and the lower-order
parts of the initial assumptions.

All norms in this paper are over $\T^2$ unless explicitly stated
otherwise.  Since $(p,u,U)$ is fixed and smooth, the same assumptions
remain valid, after changing $C_0$, when
$p_0^{\eps,\delta}$ and $U_0^{\eps,\delta}$ are replaced by their
differences from $p_0$ and $U_0$.

\subsection{The linear corrector}

For a vector field $z$, set
\begin{equation}
\label{eq:Bu-definition}
    B_u(z):=u\otimes z+z\otimes u.
\end{equation}
We define the corrector $(\pi,a,S)$ as the solution of
\begin{equation}
\label{eq:corrector-introduction}
\left\{
\begin{aligned}
    \eps\partial_t\pi+\Div a&=-\eps\partial_tp,\\
    \partial_ta+\Div S+\nabla\pi&=0,\\
    \delta\partial_tS+\nabla a
        &=B_u(a)-S-\delta\partial_tU
\end{aligned}
\right.
\end{equation}
with initial data
\begin{equation}
\label{eq:main-corrector-data}
    \pi(0)=p_0^{\eps,\delta}-p_0,\qquad
    a(0)=u_0^{\eps,\delta}-u_0,\qquad
    S(0)=U_0^{\eps,\delta}-U_0.
\end{equation}
The principal part of \eqref{eq:corrector-introduction} is a linear
symmetric-hyperbolic system and its lower-order coefficients are smooth.
Standard symmetric-hyperbolic theory \cite{Kato1975} therefore gives a
unique smooth corrector on $[0,T]$ for smooth initial data.  Indeed, for
each fixed positive pair $(\eps,\delta)$, higher-order linear energy
estimates involve bounded limit coefficients and prescribed smooth
forcing, so the solution extends throughout the fixed interval.  The variables
\begin{equation}
\label{eq:main-compensated-variables}
    X:=\delta\Div S+\delta\nabla\pi-a,\qquad
    Y:=\sqrt{\eps}\,\pi
       -\frac{\delta}{\sqrt{\eps}}\Div a
\end{equation}
will expose the relaxation dissipation.

The motivation for \eqref{eq:corrector-introduction} is the exact error
system \eqref{eq:unfiltered-error-system}: it retains every linear term
and both consistency forcings, and omits only the quadratic error
$v\otimes v$.  Thus its initial data agree with the entire initial
error.  The smallness of the nonlinear remainder is proved below, and
justifies this interpretation quantitatively.  The pressure component
$\pi$ includes acoustic propagation, stress transients, forcing and
effects of solenoidal velocity discrepancies through $B_u(a)$.
No spectral decomposition into oscillatory and non-oscillatory parts
is asserted.  The corrector retains the stiff parameter-dependent
principal part and is not supplied as an explicit or uniformly
nonstiff computational filter.

\begin{remark}[Pressure means]
\label{rem:pressure-means}
No zero-mean condition is imposed on $p_0^{\eps,\delta}$.
On the unit torus, put
$m_{\eps,\delta}=\int_{\T^2}p_0^{\eps,\delta}\,\dd x$.
Integration of the first relaxation and corrector equations, using
$\int_{\T^2}p(t)\,\dd x=0$, gives
\[
 \int_{\T^2}p^{\eps,\delta}(t)\,\dd x=m_{\eps,\delta},
 \qquad \int_{\T^2}\pi(t)\,\dd x=m_{\eps,\delta}.
\]
The corrected pressure and the pressure remainder therefore have zero
spatial mean.  The constant part of the approximate initial pressure is
included in $\pi$.  Simultaneously replacing
$p^{\eps,\delta}$ and $\pi$ by $p^{\eps,\delta}+c$ and $\pi+c$,
respectively, leaves the corrected pressure and its error unchanged,
for any time-independent constant shift within the admissible initial
data class.  A time-dependent constant shift is not a gauge symmetry of
the relaxation equation.
\end{remark}

\subsection{Main theorem}

\begin{theorem}[Linear corrector and corrected pressure convergence]
\label{thm:main-filtered-pressure}
Let $(p,u,U)$ be a smooth solution of
\eqref{eq:NS}--\eqref{eq:limit-stress} on $[0,T]\times\T^2$, and let
the smooth initial data for \eqref{eq:relaxation-system} satisfy
\eqref{eq:main-initial-velocity-L2}--\eqref{eq:main-initial-stress}.
There exist constants $\mu_*>0$ and $\delta_*\in(0,1]$, depending
only on $C_0$, $T$ and $M_T$, such that whenever
\begin{equation}
\label{eq:main-upper-parameter-condition}
    0<\delta\leq\delta_*,\qquad 0<\eps\leq\mu_*\delta
\end{equation}
the corresponding relaxation solution exists on $[0,T]$.

Moreover, the corrector defined by
\eqref{eq:corrector-introduction}--\eqref{eq:main-corrector-data}
satisfies
\begin{equation}
\label{eq:main-corrector-estimate}
\begin{aligned}
 &\sup_{0\leq t\leq T}
 \left[
 \begin{aligned}
   &\eps\norm{\pi(t)}{L^2}^2+\norm{a(t)}{L^2}^2
     +\delta\norm{S(t)}{L^2}^2\\
   &\quad+\norm{X(t)}{L^2}^2+\norm{Y(t)}{L^2}^2
     +\delta\norm{\nabla a(t)}{L^2}^2
 \end{aligned}
 \right]\\
 &\qquad
 +\int_0^T
 \bigl(\norm{S(t)}{L^2}^2+\norm{\nabla a(t)}{L^2}^2\bigr)\,\dd t
 \leq C_T\delta.
\end{aligned}
\end{equation}
In addition,
\begin{equation}
\label{eq:main-quadratic-corrector}
    \int_0^T\norm{a(t)\otimes a(t)}{L^2}^2\,\dd t
    \leq C_T\delta^2.
\end{equation}

Define the corrected remainder by
\begin{equation}
\label{eq:main-remainder-definition}
    r=p^{\eps,\delta}-p-\pi,\qquad
    w=u^{\eps,\delta}-u-a,\qquad
    Z=U^{\eps,\delta}-U-S.
\end{equation}
Then
\begin{equation}
\label{eq:main-remainder-estimate}
\begin{aligned}
 &\sup_{0\leq t\leq T}
 \left(
    \eps\norm{r(t)}{L^2}^2
    +\norm{w(t)}{L^2}^2
    +\delta\norm{Z(t)}{L^2}^2
 \right)
 +\int_0^T\norm{Z(t)}{L^2}^2\,\dd t
 \leq C_T\delta^2.
\end{aligned}
\end{equation}
Consequently,
\begin{equation}
\label{eq:main-filtered-rates}
\begin{aligned}
    \norm{p^{\eps,\delta}-\pi-p}
         {L^\infty(0,T;L^2)}
      &\leq C_T\frac{\delta}{\sqrt{\eps}},\\
    \norm{u^{\eps,\delta}-a-u}
         {L^\infty(0,T;L^2)}
      &\leq C_T\delta.
\end{aligned}
\end{equation}
If, in addition,
\begin{equation}
\label{eq:main-lower-parameter-condition}
    \frac{\delta^2}{\eps}\longrightarrow0,
\end{equation}
along any family with $\delta\to0$ satisfying
\eqref{eq:main-upper-parameter-condition}, then
\begin{equation}
\label{eq:main-pressure-convergence}
    p^{\eps,\delta}-\pi\longrightarrow p
    \quad\text{in }L^\infty(0,T;L^2(\T^2)).
\end{equation}
\end{theorem}

\begin{remark}
\label{rem:main-large-perturbation}
The expression ``large perturbation'' refers to derivative-level
behavior.  Assumption \eqref{eq:main-initial-velocity-L2} still imposes
an $L^2$ discrepancy of order $\delta^{1/2}$, whereas
\eqref{eq:main-initial-curl-div} allows derivatives of the approximate
velocity to remain of order one or to grow as the parameters vanish.
\end{remark}

\begin{remark}
\label{rem:main-parameter-window}
The two parameter restrictions have different roles.  The upper bound
$\eps\leq\mu_*\delta$ is inherited from the uniform velocity theory.
The lower condition $\delta^2/\eps\to0$ converts the weighted estimate
$\eps\norm{r}{L^2}^2\leq C_T\delta^2$ into strong unweighted pressure
convergence.  Thus the corrected-pressure convergence window is
\[
    \delta^2\ll\eps\leq\mu_*\delta.
\]
For example, if $\eps=c\delta^\alpha$ with fixed $c>0$, then
$1<\alpha<2$ is admissible for sufficiently small $\delta$; the
endpoint $\alpha=1$ is also admissible
when $c\leq\mu_*$.  For these choices the corrected pressure rate is
$C_Tc^{-1/2}\delta^{1-\alpha/2}$.  This is a sufficient window from
the present estimate, not a necessary convergence regime.  At
$\eps\asymp\delta^2$ the estimate gives only an order-one bound;
it proves neither convergence nor nonconvergence.  Sharpness of the
rate $\delta/\sqrt{\eps}$ is not established.
\end{remark}

\begin{remark}
\label{rem:main-filtered-only}
Theorem~\ref{thm:main-filtered-pressure} concerns the pressure after
subtraction of the full linear corrector.  It does not prove convergence
of $p^{\eps,\delta}$ itself under the present initial assumptions.
Proposition~\ref{prop:positive-time-obstruction} below gives a mean-zero
family for which the uncorrected pressure fails to converge even on
$[\tau,T]$ for every fixed $0<\tau<T$.  If one separately proves
$\pi\to0$ in $L^\infty(0,T;L^2)$ for a more prepared family, then
\eqref{eq:main-pressure-convergence} implies uncorrected pressure
convergence.  We do not derive a sufficient preparation criterion in
terms of the original initial data for that additional conclusion;
this is an extension of the present result.
\end{remark}

\subsection{Proof strategy}

The proof has three layers.  First, we use the two-dimensional theorem
of \cite{HuangRohdeZhang2026}.  We cite its existence and curl--divergence conclusions and derive
the needed weighted error estimate directly from the exact error energy.  Second, the basic corrector energy is
supplemented by the variables $X$ and $Y$.
Their equations contain singular acoustic terms of size
$\eps^{-1/2}$, but these terms cancel exactly in physical space.  After
one time integration by parts in the linearized convection term, the
resulting compensated functional controls
$\int_0^T\norm{\nabla a}{L^2}^2\,\dd t$.  Finally, the corrected
remainder has zero initial data.  Its only source independent of the
remainder is $a\otimes a$, whose squared $L^2$ norm is integrable with
size $O(\delta^2)$.  A closed energy estimate then proves
\eqref{eq:main-remainder-estimate}.

\section{Preliminaries and baseline estimates}
\label{sec:preliminaries}

Throughout the paper, $C$ denotes a positive constant independent of
$\eps$ and $\delta$.  The notation $C_T$ permits dependence on $T$,
$C_0$ and $M_T$ defined in \eqref{eq:limit-norm-constant}.  Constants may
change from line to line.

\subsection{Notation and two-dimensional inequalities}

For $z=(z_1,z_2)$, our conventions are
\[
    (\nabla z)_{ij}=\partial_jz_i,\qquad
    \Div z=\partial_1z_1+\partial_2z_2,\qquad
    \Curl z=\partial_1z_2-\partial_2z_1.
\]
For a matrix field $A=(A_{ij})$, the divergence is taken row-wise:
$(\Div A)_i=\sum_j\partial_jA_{ij}$.  The Frobenius product is
$A:B=\sum_{i,j}A_{ij}B_{ij}$, and $\ip{\cdot}{\cdot}$ denotes the
$L^2$ pairing.  Periodicity gives
\begin{equation}
\label{eq:periodic-integration-by-parts}
    \ip{\Div A}{z}=-\ip{A}{\nabla z},\qquad
    \ip{\nabla f}{z}=-\ip{f}{\Div z}.
\end{equation}
For vector fields $b,c$, $(b\otimes c)_{ij}=b_ic_j$, so that
\begin{equation}
\label{eq:tensor-divergence}
    \Div(b\otimes c)=(c\cdot\nabla)b+b\,\Div c.
\end{equation}
Since $u$ is smooth,
\begin{equation}
\label{eq:Bu-basic-estimates}
    \norm{B_u(z)}{L^2}\leq C_T\norm{z}{L^2},\qquad
    \norm{\Div B_u(z)}{L^2}
       \leq C_T\bigl(\norm{z}{L^2}+\norm{\nabla z}{L^2}\bigr).
\end{equation}

\begin{lemma}[Hodge identity]
\label{lem:Hodge-identity}
For every $z\in H^1(\T^2;\R^2)$,
\begin{equation}
\label{eq:Hodge-identity}
    \norm{\nabla z}{L^2}^2
    =\norm{\Div z}{L^2}^2+\norm{\Curl z}{L^2}^2.
\end{equation}
Consequently,
\begin{equation}
\label{eq:Hodge-H1-estimate}
    \norm{z}{H^1}^2
    \leq C\left(
       \norm{z}{L^2}^2+\norm{\Div z}{L^2}^2
       +\norm{\Curl z}{L^2}^2
    \right).
\end{equation}
\end{lemma}

\begin{proof}
Expanding the right-hand side of \eqref{eq:Hodge-identity} gives
\[
\begin{aligned}
 &\int_{\T^2}
  \bigl((\partial_1z_1)^2+(\partial_2z_2)^2
       +(\partial_1z_2)^2+(\partial_2z_1)^2\bigr)\,\dd x\\
 &\quad
 +2\int_{\T^2}
  \bigl(\partial_1z_1\,\partial_2z_2
       -\partial_1z_2\,\partial_2z_1\bigr)\,\dd x.
\end{aligned}
\]
The first integral is $\norm{\nabla z}{L^2}^2$.  The two mixed
integrals cancel because two integrations by parts give
$\int\partial_1z_1\,\partial_2z_2
=\int\partial_2z_1\,\partial_1z_2$.  This proves
\eqref{eq:Hodge-identity}; adding $\norm{z}{L^2}^2$ proves
\eqref{eq:Hodge-H1-estimate}.
\end{proof}

\begin{lemma}[Ladyzhenskaya inequality]
\label{lem:Ladyzhenskaya}
There exists $C>0$, depending only on $\T^2$, such that
\begin{equation}
\label{eq:Ladyzhenskaya}
    \norm{z}{L^4}^2
    \leq C\norm{z}{L^2}\norm{z}{H^1}
\end{equation}
for every $z\in H^1(\T^2)$.  Hence
\begin{equation}
\label{eq:Ladyzhenskaya-fourth-power}
    \norm{z}{L^4}^4
    \leq C\norm{z}{L^2}^2
       \bigl(\norm{z}{L^2}^2+\norm{\nabla z}{L^2}^2\bigr).
\end{equation}
If $\norm{z}{H^1}\leq M$, then
\begin{equation}
\label{eq:Ladyzhenskaya-uniform-H1}
    \norm{z}{L^4}^4\leq CM^2\norm{z}{L^2}^2.
\end{equation}
\end{lemma}

These standard inequalities may be found, for example, in
\cite{Ladyzhenskaya1969,AdamsFournier2003}.

We shall also use
\begin{equation}
\label{eq:tensor-product-L2}
    \norm{b\otimes c}{L^2}
    \leq\norm{b}{L^4}\norm{c}{L^4}.
\end{equation}

\subsection{The exact unfiltered error system}

Set
\begin{equation}
\label{eq:unfiltered-error-definition}
    \theta=p^{\eps,\delta}-p,\qquad
    v=u^{\eps,\delta}-u,\qquad
    R=U^{\eps,\delta}-U.
\end{equation}
Subtracting \eqref{eq:limit-momentum-stress} and the identity
$U=u\otimes u-\nabla u$ from \eqref{eq:relaxation-system} gives
\begin{equation}
\label{eq:unfiltered-error-system}
\left\{
\begin{aligned}
    \eps\partial_t\theta+\Div v
       &=-\eps\partial_tp,\\
    \partial_tv+\Div R+\nabla\theta&=0,\\
    \delta\partial_tR+\nabla v
       &=B_u(v)+v\otimes v-R-\delta\partial_tU.
\end{aligned}
\right.
\end{equation}
This system is exact.  Its natural weighted energy is
\begin{equation}
\label{eq:unfiltered-weighted-energy}
    \mathcal E_0(t)
    :=\frac{\eps}{2}\norm{\theta(t)}{L^2}^2
      +\frac12\norm{v(t)}{L^2}^2
      +\frac{\delta}{2}\norm{R(t)}{L^2}^2.
\end{equation}
Testing the three equations by $\theta$, $v$, and $R$ and using
\eqref{eq:periodic-integration-by-parts}, the singular first-order
terms cancel and one obtains
\begin{equation}
\label{eq:unfiltered-weighted-energy-identity}
\begin{aligned}
    \frac{\dd}{\dd t}\mathcal E_0+\norm{R}{L^2}^2
    ={}&\ip{B_u(v)+v\otimes v}{R}\\
       &-\eps\ip{\partial_tp}{\theta}
        -\delta\ip{\partial_tU}{R}.
\end{aligned}
\end{equation}
The term $v\otimes v$ is the reason why
\eqref{eq:unfiltered-weighted-energy-identity} alone does not close the
derivative-level large-perturbation argument.

\subsection{Baseline large-perturbation estimates}

We use the stated existence and curl--divergence conclusions of
\cite[Theorem~2.8]{HuangRohdeZhang2026}.  The weighted error bound
below is then proved here using those conclusions.

\begin{proposition}[Baseline estimates]
\label{prop:baseline-large-perturbation}
Suppose that
\eqref{eq:main-initial-velocity-L2}--\eqref{eq:main-initial-stress} hold.
There exist $\mu_*>0$ and $\delta_*\in(0,1]$, depending only on
$C_0$, $T$ and $M_T$, such that for
$0<\delta\leq\delta_*$ and $0<\eps\leq\mu_*\delta$, the solution of
\eqref{eq:relaxation-system} exists on $[0,T]$.  Moreover,
\begin{equation}
\label{eq:baseline-weighted-error}
    \sup_{0\leq t\leq T}
    \left(
       \eps\norm{\theta(t)}{L^2}^2
       +\norm{v(t)}{L^2}^2
       +\delta\norm{R(t)}{L^2}^2
    \right)
    \leq C_T\delta,
\end{equation}
and
\begin{equation}
\label{eq:baseline-curl-div}
\begin{aligned}
 \sup_{0\leq t\leq T}
 \left(
   \norm{\Curl u^{\eps,\delta}(t)}{L^2}^2
   +\delta\norm{\Curl u^{\eps,\delta}(t)}{H^1}^2
 \right)&\leq C_T,\\
 \sup_{0\leq t\leq T}
 \left(
   \norm{\Div u^{\eps,\delta}(t)}{L^2}^2
   +\delta\norm{\Div u^{\eps,\delta}(t)}{H^1}^2
 \right)&\leq C_T\frac{\eps}{\delta}.
\end{aligned}
\end{equation}
\end{proposition}

\begin{proof}
The existence assertion and \eqref{eq:baseline-curl-div} are the
two-dimensional conclusions of
\cite[Theorem~2.8]{HuangRohdeZhang2026}; the initial assumptions
\eqref{eq:main-initial-velocity-L2}--\eqref{eq:main-initial-stress}
are its assumptions (2.20)--(2.21).  The limit norms required there are
bounded by $M_T$.  We prove \eqref{eq:baseline-weighted-error} directly.

Let $\overline v(t)=\int_{\T^2}v(t,x)\,\dd x$.  Integration of the
momentum error equation gives $\overline v(t)=\overline v(0)$, and
\eqref{eq:main-initial-velocity-L2} implies
$|\overline v(0)|\leq C\sqrt{\delta}$.  By the Hodge identity,
\eqref{eq:baseline-curl-div} and $\eps/\delta\leq\mu_*$,
\[
 \sup_{t\leq T}\|\nabla v(t)\|_{L^2}^2
 \leq C_T\left(1+\frac{\eps}{\delta}\right)\leq C_T.
\]
The periodic Poincar\'e inequality applied to $v-\overline v$ then
gives
\begin{equation}
\label{eq:baseline-H1-before-energy}
 \sup_{t\leq T}\|v(t)\|_{H^1}\leq C_T.
\end{equation}
This bound uses no weighted pressure or stress estimate.

The Ladyzhenskaya inequality now yields
\[
 \|v\otimes v\|_{L^2}^2=\|v\|_{L^4}^4
 \leq C\|v\|_{L^2}^2\|v\|_{H^1}^2
 \leq C_T\|v\|_{L^2}^2.
\]
In \eqref{eq:unfiltered-weighted-energy-identity}, Young's inequality
therefore gives
\begin{align*}
 |\langle B_u(v)+v\otimes v,R\rangle|
 &\leq\tfrac14\|R\|_{L^2}^2+C_T\|v\|_{L^2}^2,\\
 \eps|\langle\partial_tp,\theta\rangle|
 &\leq\tfrac{\eps}{2}\|\theta\|_{L^2}^2+C_T\eps,\\
 \delta|\langle\partial_tU,R\rangle|
 &\leq\tfrac14\|R\|_{L^2}^2+C_T\delta^2.
\end{align*}
It follows that
\[
 \mathcal E_0'(t)+\tfrac12\|R(t)\|_{L^2}^2
 \leq C_T\mathcal E_0(t)+C_T(\eps+\delta^2).
\]
The initial assumptions and the fixed smooth limit data imply
$\mathcal E_0(0)\leq C\delta$.  Since
$\eps\leq\mu_*\delta$ and $\delta\leq1$, Gronwall's inequality
gives $\sup_{t\leq T}\mathcal E_0(t)\leq C_T\delta$, which proves
\eqref{eq:baseline-weighted-error}.  Thus the only external input is
the stated baseline theorem, not internal steps of its proof.
\end{proof}

\begin{corollary}[Uniform first-order velocity bound]
\label{cor:baseline-uniform-H1}
Under the assumptions of
Proposition~\ref{prop:baseline-large-perturbation},
\begin{equation}
\label{eq:baseline-uniform-H1}
    \sup_{0\leq t\leq T}
    \left(
       \norm{u^{\eps,\delta}(t)}{H^1}
       +\norm{v(t)}{H^1}
    \right)\leq C_T.
\end{equation}
\end{corollary}

\begin{proof}
By \eqref{eq:Hodge-identity}, \eqref{eq:baseline-curl-div}, and
$\eps/\delta\leq\mu_*$,
\[
    \norm{\nabla u^{\eps,\delta}}{L^2}^2
    =\norm{\Curl u^{\eps,\delta}}{L^2}^2
     +\norm{\Div u^{\eps,\delta}}{L^2}^2
    \leq C_T.
\]
Furthermore, \eqref{eq:baseline-weighted-error} and smoothness of $u$
give $\norm{u^{\eps,\delta}}{L^2}
\leq\norm{u}{L^2}+\norm{v}{L^2}\leq C_T$.  This proves the first
bound.  Since $v=u^{\eps,\delta}-u$, the second follows.
\end{proof}

\begin{remark}
\label{rem:weighted-pressure-only}
The baseline estimate gives only
\[
    \sqrt{\eps}\,
    \norm{p^{\eps,\delta}-p}{L^\infty(0,T;L^2)}
    \leq C_T\delta^{1/2}.
\]
It therefore does not imply strong pressure convergence.
\end{remark}

\subsection{Initial size of the corrector}

For brevity write
\[
    \pi_0=p_0^{\eps,\delta}-p_0,\qquad
    a_0=u_0^{\eps,\delta}-u_0,\qquad
    S_0=U_0^{\eps,\delta}-U_0.
\]
The initial compensated variables are
\begin{equation}
\label{eq:initial-compensated-variables}
    X_0=\delta\Div S_0+\delta\nabla\pi_0-a_0,\qquad
    Y_0=\sqrt{\eps}\,\pi_0
       -\frac{\delta}{\sqrt{\eps}}\Div a_0.
\end{equation}

\begin{lemma}[Initial corrector bounds]
\label{lem:initial-corrector-bounds}
Under \eqref{eq:main-initial-velocity-L2}--
\eqref{eq:main-initial-stress} and $\eps\leq\mu_*\delta$,
\begin{equation}
\label{eq:initial-corrector-basic-bound}
    \eps\norm{\pi_0}{L^2}^2+\norm{a_0}{L^2}^2
    +\delta\norm{S_0}{L^2}^2\leq C\delta
\end{equation}
and
\begin{equation}
\label{eq:corrector-initial-bound}
    \norm{X_0}{L^2}^2+\norm{Y_0}{L^2}^2
    +\delta\norm{\nabla a_0}{L^2}^2\leq C\delta.
\end{equation}
\end{lemma}

\begin{proof}
The smooth limit data are independent of the small parameters.  Using
\eqref{eq:main-initial-velocity-L2} together with
\eqref{eq:main-initial-pressure} and
\eqref{eq:main-initial-stress}, we obtain
\[
    \eps\norm{\pi_0}{L^2}^2+\norm{a_0}{L^2}^2
    +\delta\norm{S_0}{L^2}^2
    \leq C(\delta+\eps)\leq C\delta.
\]
Next,
\[
    \delta^2\norm{\Div S_0}{L^2}^2
    +\delta^2\norm{\nabla\pi_0}{L^2}^2
    +\norm{a_0}{L^2}^2\leq C\delta.
\]
Indeed, the first term follows from the stress bound
\[
    \delta\norm{U_0^{\eps,\delta}}{H^1}^2\leq C_0,
\]
whereas the second follows from the pressure bound
\[
    \delta^2\norm{p_0^{\eps,\delta}}{H^1}^2\leq C_0\delta.
\]
The corresponding limit terms are smooth.  Hence
$\norm{X_0}{L^2}^2\leq C\delta$.

Since $\Div u_0=0$, one has $\Div a_0=\Div u_0^{\eps,\delta}$.
Consequently,
\[
\begin{aligned}
    \norm{Y_0}{L^2}^2
    &\leq2\eps\norm{\pi_0}{L^2}^2
      +2\frac{\delta^2}{\eps}\norm{\Div a_0}{L^2}^2\\
    &\leq C\delta
      +C\frac{\delta^2}{\eps}\frac{\eps}{\delta}
    \leq C\delta.
\end{aligned}
\]
Finally, the Hodge identity and
\eqref{eq:main-initial-curl-div} give
\[
\begin{aligned}
    \delta\norm{\nabla a_0}{L^2}^2
    &=\delta\norm{\Curl a_0}{L^2}^2
      +\delta\norm{\Div a_0}{L^2}^2\\
    &\leq C\delta+C\eps\leq C\delta.
\end{aligned}
\]
Combining these estimates proves \eqref{eq:corrector-initial-bound}.
\end{proof}

\section{The linear corrector}
\label{sec:corrector}

\subsection{Basic weighted energy}

\begin{lemma}[Basic corrector energy]
\label{lem:corrector-basic-energy}
Let $(\pi,a,S)$ solve \eqref{eq:corrector-introduction}.  Define
\begin{equation}
\label{eq:corrector-basic-energy}
    \mathcal D(t)
    :=\frac{\eps}{2}\norm{\pi(t)}{L^2}^2
      +\frac12\norm{a(t)}{L^2}^2
      +\frac{\delta}{2}\norm{S(t)}{L^2}^2.
\end{equation}
Then
\begin{equation}
\label{eq:corrector-basic-identity}
\begin{aligned}
    \frac{\dd}{\dd t}\mathcal D+\norm{S}{L^2}^2
    =\ip{B_u(a)}{S}
      -\eps\ip{\partial_tp}{\pi}
      -\delta\ip{\partial_tU}{S}.
\end{aligned}
\end{equation}
In particular,
\begin{equation}
\label{eq:corrector-basic-differential}
    \frac{\dd}{\dd t}\mathcal D
    +\frac12\norm{S}{L^2}^2
    \leq C_T\mathcal D+C_T(\eps+\delta^2).
\end{equation}
\end{lemma}

\begin{proof}
Take the $L^2$ products of the three equations in
\eqref{eq:corrector-introduction} with $\pi$, $a$, and $S$,
respectively.  The pressure--velocity terms cancel because
\[
    \ip{\Div a}{\pi}+\ip{\nabla\pi}{a}=0,
\]
and the velocity--stress terms cancel because
$\ip{\Div S}{a}+\ip{\nabla a}{S}=0$.  The damping term contributes
$\norm{S}{L^2}^2$, which proves
\eqref{eq:corrector-basic-identity}.

Since $u$, $p$, and $U$ are smooth,
\[
\begin{aligned}
    |\ip{B_u(a)}{S}|
        &\leq\frac14\norm{S}{L^2}^2+C_T\norm{a}{L^2}^2,\\
    \eps|\ip{\partial_tp}{\pi}|
        &\leq C_T\eps+\frac{\eps}{2}\norm{\pi}{L^2}^2,\\
    \delta|\ip{\partial_tU}{S}|
        &\leq\frac14\norm{S}{L^2}^2+C_T\delta^2.
\end{aligned}
\]
Substitution in \eqref{eq:corrector-basic-identity} gives
\eqref{eq:corrector-basic-differential}.
\end{proof}

\begin{corollary}
\label{cor:corrector-basic-bound}
Under the assumptions of Theorem~\ref{thm:main-filtered-pressure},
\begin{equation}
\label{eq:corrector-basic-bound}
    \sup_{0\leq t\leq T}\mathcal D(t)
    +\int_0^T\norm{S(t)}{L^2}^2\,\dd t
    \leq C_T\delta.
\end{equation}
\end{corollary}

\begin{proof}
By Lemma~\ref{lem:initial-corrector-bounds},
$\mathcal D(0)\leq C\delta$.  Since
$\eps\leq\mu_*\delta$ and $\delta\leq1$, integration of
\eqref{eq:corrector-basic-differential} followed by Gronwall's
inequality proves the assertion.
\end{proof}

\subsection{Evolution of the compensated variables}

\begin{lemma}[Compensated equations]
\label{lem:compensated-equations}
The variables in \eqref{eq:main-compensated-variables} satisfy
\begin{equation}
\label{eq:X-algebraic-identity}
    X=-a-\delta\partial_ta
\end{equation}
and the coupled system
\begin{equation}
\label{eq:XY-system}
\left\{
\begin{aligned}
    \partial_tX+\Delta a-\frac1{\sqrt{\eps}}\nabla Y
        &=\Div B_u(a)+\delta\partial_{tt}u,\\
    \partial_tY-\frac1{\sqrt{\eps}}\Div X
        &=-\sqrt{\eps}\,\partial_tp.
\end{aligned}
\right.
\end{equation}
\end{lemma}

\begin{proof}
The second corrector equation gives
$\Div S+\nabla\pi=-\partial_ta$, and hence
\[
    X=\delta(\Div S+\nabla\pi)-a
      =-a-\delta\partial_ta.
\]
To compute $\partial_tX$, take the divergence of the third corrector
equation:
\[
    \delta\Div\partial_tS+\Delta a
    =\Div B_u(a)-\Div S-\delta\Div\partial_tU.
\]
Using the first two corrector equations, we obtain
\[
\begin{aligned}
    \partial_tX
    &=
      \delta\Div\partial_tS+\delta\nabla\partial_t\pi
      -\partial_ta\\
    &=\Div B_u(a)-\Delta a+\nabla\pi
      -\frac{\delta}{\eps}\nabla\Div a
      -\delta\bigl(\Div\partial_tU+\nabla\partial_tp\bigr).
\end{aligned}
\]
Differentiating \eqref{eq:limit-momentum-stress} in time gives
$\Div\partial_tU+\nabla\partial_tp=-\partial_{tt}u$.  Moreover,
\[
    \nabla\pi-\frac{\delta}{\eps}\nabla\Div a
    =\frac1{\sqrt{\eps}}\nabla Y.
\]
These identities yield the first equation in \eqref{eq:XY-system}.

For the second equation, differentiate the definition of $Y$ and use
the first two corrector equations:
\[
\begin{aligned}
    \partial_tY
    &=-\frac1{\sqrt{\eps}}\Div a-\sqrt{\eps}\,\partial_tp
      -\frac{\delta}{\sqrt{\eps}}\Div\partial_ta\\
    &=\frac1{\sqrt{\eps}}
       \bigl(\delta\Div\Div S+\delta\Delta\pi-\Div a\bigr)
      -\sqrt{\eps}\,\partial_tp\\
    &=\frac1{\sqrt{\eps}}\Div X
      -\sqrt{\eps}\,\partial_tp.
\end{aligned}
\]
This completes the proof.
\end{proof}

\subsection{Exact compensated identity and gradient dissipation}

Define
\begin{equation}
\label{eq:compensated-functional}
\begin{aligned}
    \mathcal K(t)
    :={}&\frac12\norm{X(t)}{L^2}^2
        +\frac12\norm{Y(t)}{L^2}^2
        +\frac{\delta}{2}\norm{\nabla a(t)}{L^2}^2\\
       &-\delta\ip{B_u(a(t))}{\nabla a(t)}.
\end{aligned}
\end{equation}

\begin{lemma}[Exact compensated identity]
\label{lem:exact-compensated-identity}
For every $t\in[0,T]$,
\begin{equation}
\label{eq:exact-compensated-identity}
\begin{aligned}
 \frac{\dd}{\dd t}\mathcal K+\norm{\nabla a}{L^2}^2
 ={}&
  2\ip{B_u(a)}{\nabla a}
  +\ip{B_u(X)}{\nabla a}
  -\delta\ip{B_{\partial_tu}(a)}{\nabla a}\\
 &+\delta\ip{\partial_{tt}u}{X}
  -\sqrt{\eps}\ip{\partial_tp}{Y}.
\end{aligned}
\end{equation}
\end{lemma}

\begin{proof}
Test the two equations in \eqref{eq:XY-system} by $X$ and $Y$.
The singular acoustic terms cancel exactly:
\[
    -\frac1{\sqrt{\eps}}\ip{\nabla Y}{X}
    -\frac1{\sqrt{\eps}}\ip{\Div X}{Y}=0.
\]
We are left with
\begin{equation}
\label{eq:XY-first-energy}
\begin{aligned}
 \frac12\frac{\dd}{\dd t}
   \bigl(\norm{X}{L^2}^2+\norm{Y}{L^2}^2\bigr)
 +\ip{\Delta a}{X}
 ={}&\ip{\Div B_u(a)}{X}\\
 &+\delta\ip{\partial_{tt}u}{X}
  -\sqrt{\eps}\ip{\partial_tp}{Y}.
\end{aligned}
\end{equation}
The algebraic identity \eqref{eq:X-algebraic-identity} gives
\begin{equation}
\label{eq:laplacian-compensation}
    \ip{\Delta a}{X}
    =\norm{\nabla a}{L^2}^2
     +\frac{\delta}{2}\frac{\dd}{\dd t}
       \norm{\nabla a}{L^2}^2.
\end{equation}
It remains to treat the convection term without differentiating $X$
in space.  Again using $X=-a-\delta\partial_ta$ and periodic
integration by parts,
\begin{equation}
\label{eq:convection-first-rewrite}
\begin{aligned}
    \ip{\Div B_u(a)}{X}
    &=\ip{B_u(a)}{\nabla a}
      +\delta\ip{B_u(a)}{\nabla\partial_ta}\\
    &=\delta\frac{\dd}{\dd t}
       \ip{B_u(a)}{\nabla a}
      +\ip{B_u(a)}{\nabla a}
      -\delta\ip{\partial_tB_u(a)}{\nabla a}.
\end{aligned}
\end{equation}
Since
\[
    \partial_tB_u(a)=B_{\partial_tu}(a)+B_u(\partial_ta)
\]
and $\delta\partial_ta=-a-X$, one has
\[
    -\delta B_u(\partial_ta)=B_u(a)+B_u(X).
\]
Substituting this relation into
\eqref{eq:convection-first-rewrite} yields
\begin{equation}
\label{eq:convection-exact-rewrite}
\begin{aligned}
    \ip{\Div B_u(a)}{X}
    ={}&\delta\frac{\dd}{\dd t}
       \ip{B_u(a)}{\nabla a}
       +2\ip{B_u(a)}{\nabla a}\\
      &+\ip{B_u(X)}{\nabla a}
       -\delta\ip{B_{\partial_tu}(a)}{\nabla a}.
\end{aligned}
\end{equation}
Combining \eqref{eq:XY-first-energy},
\eqref{eq:laplacian-compensation}, and
\eqref{eq:convection-exact-rewrite}, and moving the total time
derivative in the first line of
\eqref{eq:convection-exact-rewrite} to the left-hand side, proves
\eqref{eq:exact-compensated-identity}.
\end{proof}

\begin{proposition}[Compensated corrector estimate]
\label{prop:compensated-corrector}
Under the assumptions of Theorem~\ref{thm:main-filtered-pressure},
\begin{equation}
\label{eq:compensated-corrector-estimate}
\begin{aligned}
 &\sup_{0\leq t\leq T}
 \left[
 \begin{aligned}
   &\eps\norm{\pi(t)}{L^2}^2+\norm{a(t)}{L^2}^2
     +\delta\norm{S(t)}{L^2}^2\\
   &\quad+\norm{X(t)}{L^2}^2+\norm{Y(t)}{L^2}^2
     +\delta\norm{\nabla a(t)}{L^2}^2
 \end{aligned}
 \right]\\
 &\qquad
 +\int_0^T
 \bigl(\norm{S(t)}{L^2}^2+\norm{\nabla a(t)}{L^2}^2\bigr)\,\dd t
 \leq C_T\delta.
\end{aligned}
\end{equation}
In particular,
\begin{equation}
\label{eq:corrector-uniform-H1}
    \sup_{0\leq t\leq T}\norm{a(t)}{H^1}\leq C_T.
\end{equation}
\end{proposition}

\begin{proof}
We first establish coercivity.  Smoothness of $u$ and Young's
inequality imply
\[
    \delta|\ip{B_u(a)}{\nabla a}|
    \leq\frac{\delta}{4}\norm{\nabla a}{L^2}^2
        +C_T\delta\norm{a}{L^2}^2.
\]
It follows that, for a sufficiently large fixed $\Lambda>0$, the
combined functional
\begin{equation}
\label{eq:total-compensated-energy}
    \mathcal H:=\mathcal K+\Lambda\mathcal D
\end{equation}
is equivalent, uniformly for $0<\delta\leq1$, to
\begin{equation}
\label{eq:total-compensated-equivalence}
\begin{aligned}
 &\norm{X}{L^2}^2+\norm{Y}{L^2}^2
 +\delta\norm{\nabla a}{L^2}^2\\
 &\qquad
 +\eps\norm{\pi}{L^2}^2+\norm{a}{L^2}^2
 +\delta\norm{S}{L^2}^2.
\end{aligned}
\end{equation}
More precisely, there are $c_T,C_T>0$, independent of the small
parameters, such that $c_T\mathcal Q\leq\mathcal H\leq C_T\mathcal Q$,
where $\mathcal Q$ denotes the expression in
\eqref{eq:total-compensated-equivalence}.

We next estimate the right-hand side of
\eqref{eq:exact-compensated-identity}.  Since $u$ and $\partial_tu$
are smooth,
\begin{align*}
 2|\ip{B_u(a)}{\nabla a}|
    &\leq\frac18\norm{\nabla a}{L^2}^2+C_T\norm{a}{L^2}^2,\\
 |\ip{B_u(X)}{\nabla a}|
    &\leq\frac18\norm{\nabla a}{L^2}^2+C_T\norm{X}{L^2}^2,\\
 \delta|\ip{B_{\partial_tu}(a)}{\nabla a}|
    &\leq\frac18\norm{\nabla a}{L^2}^2+C_T\delta^2\norm{a}{L^2}^2.
\end{align*}
The smooth forcing terms satisfy
\[
\begin{aligned}
 \delta|\ip{\partial_{tt}u}{X}|
    &\leq C_T\norm{X}{L^2}^2+C_T\delta^2,\\
 \sqrt{\eps}|\ip{\partial_tp}{Y}|
    &\leq C_T\norm{Y}{L^2}^2+C_T\eps.
\end{aligned}
\]
Consequently,
\begin{equation}
\label{eq:K-differential-inequality}
    \frac{\dd}{\dd t}\mathcal K
    +\frac12\norm{\nabla a}{L^2}^2
    \leq C_T\bigl(
       \norm{a}{L^2}^2+\norm{X}{L^2}^2+\norm{Y}{L^2}^2
    \bigr)+C_T(\eps+\delta^2).
\end{equation}
Choose $\Lambda$ in \eqref{eq:total-compensated-energy} large enough
for coercivity and add $\Lambda$ times
\eqref{eq:corrector-basic-differential} to
\eqref{eq:K-differential-inequality}.  Using
\eqref{eq:total-compensated-equivalence} gives
\begin{equation}
\label{eq:H-differential-inequality}
    \frac{\dd}{\dd t}\mathcal H
    +c\norm{\nabla a}{L^2}^2+c\norm{S}{L^2}^2
    \leq C_T\mathcal H+C_T(\eps+\delta^2)
\end{equation}
for some $c>0$ independent of $\eps,\delta$.

Lemma~\ref{lem:initial-corrector-bounds} and the coercivity argument
give $\mathcal H(0)\leq C\delta$.  Gronwall's inequality, together with
$\eps\leq\mu_*\delta$, yields
\[
    \sup_{0\leq t\leq T}\mathcal H(t)
    +\int_0^T
      \bigl(\norm{\nabla a}{L^2}^2+\norm{S}{L^2}^2\bigr)\,\dd t
    \leq C_T\delta.
\]
The equivalence \eqref{eq:total-compensated-equivalence} proves
\eqref{eq:compensated-corrector-estimate}.  Finally,
$\norm{a}{L^2}^2\leq C_T\delta$ and
$\delta\norm{\nabla a}{L^2}^2\leq C_T\delta$ imply
\eqref{eq:corrector-uniform-H1}.
\end{proof}

\begin{corollary}[Quadratic corrector bound]
\label{cor:quadratic-corrector}
Under the same assumptions,
\begin{equation}
\label{eq:quadratic-corrector-bound}
    \int_0^T\norm{a\otimes a}{L^2}^2\,\dd t
    =\int_0^T\norm{a}{L^4}^4\,\dd t
    \leq C_T\delta^2.
\end{equation}
\end{corollary}

\begin{proof}
By \eqref{eq:Ladyzhenskaya-fourth-power},
\[
    \norm{a}{L^4}^4
    \leq C\norm{a}{L^2}^2
      \bigl(\norm{a}{L^2}^2+\norm{\nabla a}{L^2}^2\bigr).
\]
Proposition~\ref{prop:compensated-corrector} gives
$\sup_t\norm{a(t)}{L^2}^2\leq C_T\delta$ and
\[
    \int_0^T
      \bigl(\norm{a}{L^2}^2+\norm{\nabla a}{L^2}^2\bigr)\,\dd t
    \leq C_T\delta.
\]
Multiplication of the last two bounds proves
\eqref{eq:quadratic-corrector-bound}.
\end{proof}

\section{Nonlinear remainder and corrected pressure convergence}
\label{sec:filtered-pressure}

\subsection{The exact remainder equations}

Recall the remainder $(r,w,Z)$ from
\eqref{eq:main-remainder-definition}.  Since the corrector initial data
are the complete initial discrepancies,
\begin{equation}
\label{eq:remainder-zero-data}
    r(0)=0,\qquad w(0)=0,\qquad Z(0)=0.
\end{equation}

\begin{lemma}[Exact remainder system]
\label{lem:exact-filtered-system}
The corrected remainder satisfies
\begin{equation}
\label{eq:filtered-remainder-system}
\left\{
\begin{aligned}
    \eps\partial_tr+\Div w&=0,\\
    \partial_tw+\Div Z+\nabla r&=0,\\
    \delta\partial_tZ+\nabla w
       &=B_u(w)+a\otimes a+a\otimes w+w\otimes a
         +w\otimes w-Z.
\end{aligned}
\right.
\end{equation}
\end{lemma}

\begin{proof}
The first two equations follow by subtracting the first two equations
of \eqref{eq:corrector-introduction} from
\eqref{eq:unfiltered-error-system}.  For the third one, use $v=a+w$
and $R=S+Z$.  Then
\[
\begin{aligned}
 B_u(v)+v\otimes v-R-\delta\partial_tU
  -\bigl(B_u(a)-S-\delta\partial_tU\bigr)
 ={}&B_u(w)+a\otimes a+a\otimes w\\
    &+w\otimes a+w\otimes w-Z.
\end{aligned}
\]
This is exactly the third equation in
\eqref{eq:filtered-remainder-system}.
\end{proof}

\subsection{A uniform first-order bound}

\begin{lemma}
\label{lem:uniform-H1-remainder}
Under the assumptions of Theorem~\ref{thm:main-filtered-pressure},
\begin{equation}
\label{eq:uniform-H1-remainder}
    \sup_{0\leq t\leq T}\norm{w(t)}{H^1}\leq C_T.
\end{equation}
\end{lemma}

\begin{proof}
By Corollary~\ref{cor:baseline-uniform-H1},
$\sup_t\norm{u^{\eps,\delta}(t)-u(t)}{H^1}\leq C_T$.
Proposition~\ref{prop:compensated-corrector} gives
$\sup_t\norm{a(t)}{H^1}\leq C_T$.  Since
$w=u^{\eps,\delta}-u-a$, the triangle inequality proves
\eqref{eq:uniform-H1-remainder}.
\end{proof}

\subsection{Nonlinear energy estimate}

Define
\begin{equation}
\label{eq:remainder-energy}
    \mathcal R(t)
    :=\frac{\eps}{2}\norm{r(t)}{L^2}^2
      +\frac12\norm{w(t)}{L^2}^2
      +\frac{\delta}{2}\norm{Z(t)}{L^2}^2.
\end{equation}

\begin{lemma}[Remainder energy identity]
\label{lem:remainder-energy-identity}
Let
\begin{equation}
\label{eq:remainder-nonlinearity}
    \mathcal N
    :=B_u(w)+a\otimes a+a\otimes w+w\otimes a+w\otimes w.
\end{equation}
Then
\begin{equation}
\label{eq:remainder-energy-identity}
    \frac{\dd}{\dd t}\mathcal R+\norm{Z}{L^2}^2
    =\ip{\mathcal N}{Z}.
\end{equation}
\end{lemma}

\begin{proof}
Test the equations in \eqref{eq:filtered-remainder-system} by $r$,
$w$, and $Z$.  Periodicity gives
\[
    \ip{\Div w}{r}+\ip{\nabla r}{w}=0,\qquad
    \ip{\Div Z}{w}+\ip{\nabla w}{Z}=0.
\]
All singular first-order terms therefore cancel, while the damping
term contributes $\norm{Z}{L^2}^2$.  This proves
\eqref{eq:remainder-energy-identity}.
\end{proof}

\begin{lemma}[Quadratic nonlinear bound]
\label{lem:remainder-nonlinear-bound}
There is a constant $C_T$, independent of the small parameters, such
that
\begin{equation}
\label{eq:remainder-nonlinear-bound}
    \norm{\mathcal N}{L^2}^2
    \leq C_T\norm{w}{L^2}^2
       +C\norm{a}{L^4}^4.
\end{equation}
\end{lemma}

\begin{proof}
The smoothness of $u$ gives
$\norm{B_u(w)}{L^2}^2\leq C_T\norm{w}{L^2}^2$.
Furthermore,
\[
    \norm{a\otimes a}{L^2}^2=\norm{a}{L^4}^4,\qquad
    \norm{w\otimes w}{L^2}^2=\norm{w}{L^4}^4.
\]
By Lemma~\ref{lem:uniform-H1-remainder} and
\eqref{eq:Ladyzhenskaya-uniform-H1},
\begin{equation}
\label{eq:w-fourth-order-bound}
    \norm{w}{L^4}^4\leq C_T\norm{w}{L^2}^2.
\end{equation}
For the mixed products, \eqref{eq:tensor-product-L2} and the elementary
inequality $AB\leq(A^2+B^2)/2$ yield
\[
\begin{aligned}
    \norm{a\otimes w}{L^2}^2
    &\leq\norm{a}{L^4}^2\norm{w}{L^4}^2\\
    &\leq\frac12\norm{a}{L^4}^4
          +\frac12\norm{w}{L^4}^4\\
    &\leq\frac12\norm{a}{L^4}^4
          +C_T\norm{w}{L^2}^2.
\end{aligned}
\]
The same estimate holds for $w\otimes a$.  Squaring the sum in
\eqref{eq:remainder-nonlinearity} and combining the preceding bounds
proves \eqref{eq:remainder-nonlinear-bound}.
\end{proof}

\begin{proposition}[Corrected remainder estimate]
\label{prop:filtered-remainder}
Under the assumptions of Theorem~\ref{thm:main-filtered-pressure},
\begin{equation}
\label{eq:filtered-remainder-estimate}
    \sup_{0\leq t\leq T}\mathcal R(t)
    +\int_0^T\norm{Z(t)}{L^2}^2\,\dd t
    \leq C_T\delta^2.
\end{equation}
\end{proposition}

\begin{proof}
By Cauchy--Young's inequality,
\[
    |\ip{\mathcal N}{Z}|
    \leq\frac12\norm{Z}{L^2}^2
        +\frac12\norm{\mathcal N}{L^2}^2.
\]
Combining this with
Lemmas~\ref{lem:remainder-energy-identity} and
\ref{lem:remainder-nonlinear-bound}, and using
$\norm{w}{L^2}^2\leq2\mathcal R$, gives
\begin{equation}
\label{eq:remainder-gronwall-inequality}
    \frac{\dd}{\dd t}\mathcal R
    +\frac12\norm{Z}{L^2}^2
    \leq C_T\mathcal R+C\norm{a}{L^4}^4.
\end{equation}
Because of the zero initial data \eqref{eq:remainder-zero-data},
$\mathcal R(0)=0$.  Integrating
\eqref{eq:remainder-gronwall-inequality}, applying Gronwall's
inequality, and then using
Corollary~\ref{cor:quadratic-corrector}, we obtain
\[
    \sup_{0\leq t\leq T}\mathcal R(t)
    \leq C_T\int_0^T\norm{a(t)}{L^4}^4\,\dd t
    \leq C_T\delta^2.
\]
Integrating \eqref{eq:remainder-gronwall-inequality} once more and
using the just-established supremum bound proves the estimate for
$\int_0^T\norm{Z}{L^2}^2\,\dd t$.
\end{proof}

\subsection{Proof of the main theorem}

\begin{proof}[Proof of Theorem~\ref{thm:main-filtered-pressure}]
The existence of the relaxation solution on $[0,T]$ follows from
Proposition~\ref{prop:baseline-large-perturbation}.  The corrector
estimate \eqref{eq:main-corrector-estimate} is
Proposition~\ref{prop:compensated-corrector}, and
\eqref{eq:main-quadratic-corrector} is
Corollary~\ref{cor:quadratic-corrector}.

By the definition \eqref{eq:remainder-energy},
Proposition~\ref{prop:filtered-remainder} gives
\[
\begin{aligned}
 &\sup_{0\leq t\leq T}
 \left(
   \eps\norm{r(t)}{L^2}^2+\norm{w(t)}{L^2}^2
   +\delta\norm{Z(t)}{L^2}^2
 \right)\\
 &\qquad
 +\int_0^T\norm{Z(t)}{L^2}^2\,\dd t
 \leq C_T\delta^2,
\end{aligned}
\]
which is \eqref{eq:main-remainder-estimate}.  Taking square roots
yields
\[
    \norm{r}{L^\infty(0,T;L^2)}
       \leq C_T\frac{\delta}{\sqrt{\eps}},\qquad
    \norm{w}{L^\infty(0,T;L^2)}
       \leq C_T\delta.
\]
Since $r=p^{\eps,\delta}-p-\pi$ and
$w=u^{\eps,\delta}-u-a$, these are precisely
\eqref{eq:main-filtered-rates}.  Finally,
$\delta^2/\eps\to0$ makes the first right-hand side tend to zero,
proving \eqref{eq:main-pressure-convergence}.
\end{proof}

\section{A positive-time obstruction for the uncorrected pressure}
\label{sec:pressure-obstruction}

The following example shows that the initial assumptions permit
mean-zero pressure oscillations with order-one amplitude at positive
times.  The construction first gives an exact linear corrector and then
uses the nonlinear remainder estimate to obtain a conclusion for the
original nonlinear relaxation system.

\begin{proposition}[Persistent mean-zero pressure modes]
\label{prop:positive-time-obstruction}
Fix $T>0$ and take the limit solution $(p,u,U)=(0,0,0)$.
There exist constants $c>0$, $b_T>0$, and a family of smooth mean-zero
initial data satisfying
\eqref{eq:main-initial-velocity-L2}--\eqref{eq:main-initial-stress}
with $\eps=c\delta$, such that, for all sufficiently small $\delta$,
the nonlinear relaxation solution exists on $[0,T]$ and
\begin{equation}
\label{eq:nonlinear-pressure-lower-bound}
 \inf_{0\leq t\leq T}\|p^{c\delta,\delta}(t)\|_{L^2}\geq b_T.
\end{equation}
Consequently, for every $0<\tau<T$, the uncorrected pressure does not
converge to zero in $L^\infty(\tau,T;L^2)$ or in
$L^2(\tau,T;L^2)$, although the corrected pressure converges to zero.
\end{proposition}

\begin{proof}
Fix $C_0=1$ and choose $c\in(0,\min\{1,\mu_*\}]$, where $\mu_*$
is the threshold in Theorem~\ref{thm:main-filtered-pressure} for the
zero limit and this value of $C_0$.  Put $k=2\pi$, and consider the
complex longitudinal ansatz
\[
 (\pi,a,S)=e^{\lambda t+ikx_1}
       (1,Ae_1,Qe_1\otimes e_1),\qquad e_1=(1,0).
\]
The zero-background corrector equations are equivalent to
\begin{equation}
\label{eq:mode-eigenvector-system}
 \eps\lambda+ikA=0,\qquad
 \lambda A+ik(Q+1)=0,\qquad
 (1+\delta\lambda)Q+ikA=0.
\end{equation}
Hence
\begin{equation}
\label{eq:mode-characteristic-polynomial}
 \eps\delta\lambda^3+\eps\lambda^2
   +(\eps+\delta)k^2\lambda+k^2=0,
\end{equation}
and, for a root with $1+\delta\lambda\ne0$,
\begin{equation}
\label{eq:mode-eigenvector}
 A=-\frac{\eps\lambda}{ik},\qquad
 Q=\frac{\eps\lambda}{1+\delta\lambda}.
\end{equation}

We justify the acoustic-root expansion explicitly.  Write
$h=\sqrt{\delta}$, $z=h\lambda$ and $\eps=ch^2$.
Equation \eqref{eq:mode-characteristic-polynomial} becomes
\[
 F(z,h):=cz^2+k^2+h\bigl(cz^3+(c+1)k^2z\bigr)=0.
\]
At $h=0$ the root $z_0=ik/\sqrt c$ is simple, since
$\partial_zF(z_0,0)=2cz_0\ne0$.  The implicit function theorem
gives a complex analytic root $z(h)$ near $z_0$.  Differentiation
at zero yields
\[
 z'(0)=-\frac{cz_0^3+(c+1)k^2z_0}{2cz_0}
        =-\frac{k^2}{2}.
\]
Thus one root, denoted by $\lambda_+(\delta)$, satisfies
\begin{equation}
\label{eq:mode-root-expansion}
 \lambda_+(\delta)
 =\frac{ik}{\sqrt{c\delta}}-\frac{k^2}{2}
       +O(\sqrt{\delta}).
\end{equation}
Here $c$ and $k$ are fixed.  In particular,
$\operatorname{Re}\lambda_+\geq-(k^2/2+1)$ for sufficiently small
$\delta$, and \eqref{eq:mode-eigenvector} gives
$|A|+|Q|\leq C\sqrt{\delta}$.

For a fixed small amplitude $\eta>0$, take the real-valued initial data
\begin{equation}
\label{eq:mode-initial-data}
\begin{aligned}
 p_0^{c\delta,\delta}(x)&=\eta\operatorname{Re}(e^{ikx_1}),\\
 u_0^{c\delta,\delta}(x)&=\eta\operatorname{Re}(Ae^{ikx_1})e_1,\\
 U_0^{c\delta,\delta}(x)&=\eta\operatorname{Re}(Qe^{ikx_1})
                                      e_1\otimes e_1.
\end{aligned}
\end{equation}
All fields are smooth and have zero spatial mean.  Since $k$ is fixed,
for every fixed integer $m$ their $H^m$ norms are bounded by fixed
multiples of $\eta$, $\eta\sqrt{\delta}$ and
$\eta\sqrt{\delta}$, respectively.  More explicitly,
\begin{align*}
 \|u_0^{c\delta,\delta}\|_{L^2}^2&\leq C\eta^2\delta,
 &\Curl u_0^{c\delta,\delta}&=0,\\
 \|\Div u_0^{c\delta,\delta}\|_{L^2}^2
 +\delta\|\Div u_0^{c\delta,\delta}\|_{H^1}^2
 &\leq C\eta^2\delta,\\
 c\delta\|p_0^{c\delta,\delta}\|_{L^2}^2
 +\delta^2\|p_0^{c\delta,\delta}\|_{H^1}^2
 +\delta^3\|p_0^{c\delta,\delta}\|_{H^2}^2
 &\leq C\eta^2\delta,\\
 \|U_0^{c\delta,\delta}\|_{L^2}^2
 +\delta\|U_0^{c\delta,\delta}\|_{H^1}^2
 +\delta^2\|\Div U_0^{c\delta,\delta}\|_{H^1}^2
 &\leq C\eta^2\delta.
\end{align*}
The constants here can depend on $c$ and $k$, which are already fixed.
Choose $\eta$ small enough, and then $\delta$ small enough, that the
velocity and pressure bounds hold with $C_0=1$, the divergence bound
is at most $\eps/\delta=c$, and the stress bound is at most one.
Thus all the main initial assumptions hold with the previously fixed
$C_0$; the choice of $c$ is not circular.

By linearity and uniqueness, the corrector with these initial data is
exactly the real part of the eigenmode multiplied by $\eta$.
In particular,
\[
 \pi(t,x)=\eta e^{t\operatorname{Re}\lambda_+}
          \cos\bigl(kx_1+t\operatorname{Im}\lambda_+\bigr),
 \qquad
 \|\pi(t)\|_{L^2}
 =\frac{\eta}{\sqrt2}e^{t\operatorname{Re}\lambda_+}.
\]
Set
\[
 d_T=\frac{\eta}{\sqrt2}\exp\bigl(-(k^2/2+1)T\bigr)>0.
\]
The last formula implies that $\|\pi(t)\|_{L^2}\geq d_T$ for every
$t\in[0,T]$.
Since $\delta^2/\eps=\delta/c\to0$, the main theorem applies to
the nonlinear relaxation solution with the same initial data and gives
\[
 \sup_{t\leq T}\|p^{c\delta,\delta}(t)-\pi(t)\|_{L^2}
 \leq C_T c^{-1/2}\sqrt{\delta}.
\]
For sufficiently small $\delta$, this is at most $d_T/2$.
The reverse triangle inequality proves
\eqref{eq:nonlinear-pressure-lower-bound} with $b_T=d_T/2$.
The lower bounds in the two asserted positive-time norms follow,
including $\|p^{c\delta,\delta}\|_{L^2(\tau,T;L^2)}
\geq b_T\sqrt{T-\tau}$.  The same main-theorem estimate proves
convergence of $p^{c\delta,\delta}-\pi$ to zero.
\end{proof}

\begin{remark}
This example uses $\eps=c\delta$, which lies strictly inside the
sufficient convergence window for the corrected pressure.  It does not
show failure of corrected pressure convergence at
$\eps\asymp\delta^2$, and it does not prove sharpness of the nonlinear
remainder rate.  It demonstrates a distinct obstruction: an order-one
mean-zero linear pressure response can persist on fixed positive-time
intervals despite a small velocity discrepancy and a vanishing
corrected pressure error.
\end{remark}

\section{Conclusions and perspectives}
\label{sec:conclusions}

The compensated gradient estimate makes the full linear response an
accurate approximation of the nonlinear relaxation error under the
two-dimensional baseline assumptions.  The correction depends on the
exact limit solution, and the positive-time example demonstrates that
corrected and uncorrected pressure convergence are distinct conclusions.

Optimality of the sufficient window and of the rate remains unresolved.
At $\eps\asymp\delta^2$ the present estimate is inconclusive.
Further questions include preparation criteria on the original data,
corrections using only approximate fields, and a three-dimensional
extension with sufficient control of the corrector self-interaction.

\section*{Declarations}

\subsection*{Funding}

The authors received no funding for this work.

\subsection*{Competing interests}

The authors declare that they have no competing interests.

\subsection*{Data availability}

No data were generated or analyzed in this study.

\subsection*{Author contributions}

All authors accept responsibility for the entire content of this
manuscript and approve its submission.

\subsection*{Use of generative artificial intelligence}

ChatGPT (OpenAI) was used to assist with preliminary mathematical
derivations and manuscript preparation.  Its outputs were used as
working suggestions rather than as a source of mathematical authority.
The authors retain responsibility for independent verification of the
derivations, citation accuracy and the final content of the manuscript.
This verification covers subtraction of the governing equations,
periodic integration by parts, the compensated identities, the
parameter-weighted estimates and the Fourier-mode calculation.

\providecommand{\bysame}{\leavevmode\hbox to3em{\hrulefill}\thinspace}
\providecommand{\MR}{\relax\ifhmode\unskip\space\fi MR }
\providecommand{\MRhref}[2]{%
  \href{http://www.ams.org/mathscinet-getitem?mr=#1}{#2}
}
\providecommand{\href}[2]{#2}

\end{document}